\documentclass[12pt]{amsart}
\usepackage{amssymb,amsmath,amsthm,mathrsfs,multirow,enumerate}
\usepackage{enumitem,mathtools,cleveref,color}
\usepackage[T1]{fontenc}
\usepackage{charter}
\usepackage[expert]{mathdesign}

\usepackage{color}
\usepackage{graphicx} 
\usepackage[all]{xy}

\usepackage{blkarray} 

\usepackage[bottom=1.3in, top=1.3in, left=1.3in, right=1.3in]{geometry} 

\numberwithin{equation}{section}

\newtheorem{thm}{Theorem}[section]
\newtheorem{lem}[thm]{Lemma}

\theoremstyle{definition}

\newtheorem*{rem}{Remark}

\theoremstyle{definition} \numberwithin{equation}{section}

\newcommand{\Z}{{\mathbb{Z}}} 
\newcommand{\cha}{{\textup{char}}}

\newcommand{\td}{\ensuremath{~}}

\newcommand{\F}{\ensuremath{\mathbb{F}}}
\newcommand{\vx}{\ensuremath{\vec{x}}}
\newcommand{\vy}{\ensuremath{\vec{y}}}
\newcommand{\vz}{\ensuremath{\vec{z}}}
\newcommand{\vw}{\ensuremath{\vec{w}}}

\begin{document}

\title[Maximal orthogonal sets of unimodular vectors]
{Maximal orthogonal sets of unimodular vectors over finite local rings of odd characteristic}
\author{Songpon Sriwongsa and Siripong Sirisuk}

\address{Songpon Sriwongsa \\ Department of Mathematics \\ Faculty of Science
	\\ King Mongkut's University of Technology Thonburi \\ Bangkok, 10140 THAILAND}
\email{\tt songpon.sri@kmutt.ac.th, songponsriwongsa@gmail.com}

\address{Siripong Sirisuk \\ Department of Mathematics and Computer Science\\ Faculty of Science
\\ Chulalongkorn University\\ Bangkok, 10330 THAILAND}
\email{\tt siripong.srs@gmail.com}



\begin{abstract}
	Let $R$ be a finite local ring of odd characteristic and $\beta$ a non-degenerate symmetric bilinear form on $R^2$.
 In this short note, we determine the largest possible cardinality of pairwise orthogonal sets of unimodular vectors in $R^2$.
\end{abstract}

\date{\today}

\maketitle

\section{Introduction}

Two famous distinct distances and unit distances problems for the plane $\mathbb{R}^2$ were posed by Erd{\H o}s \cite{Erdos1946}. The first problem asks for the minimum number of distinct distances among $n$ points in the plane while the latter problem asks for the maximum number of the unit distances that can occur among $n$ points in the plane. These problems were generalized to the $n$-dimensional Euclidean space case \cite{Erdos1975}. 

Points in the $n$-dimensional vector space over $\mathbb{F}_q$ were considered for similar questions, see \cite{BKT04,IR07,ISX10} for example.  In \cite{IR07}, the authors defined a specific distance between two points in $\mathbb{F}_q^n$ where $q$ is an odd prime power and studied the distinct distances problem. It is natural to ask the two mentioned problems more generally by using an arbitrary quadratic form over $\mathbb{F}_q^n$. Recently, the problems in this direction were studied as follow. Let $\beta$ be a non-degenerate symmetric bilinear form over $\mathbb{F}_q^n$. The largest possible cardinality of a subset $S \subset \mathbb{F}_q^n$ so that $\beta(\vec{x}, \vec{y}) = 0$ for every distinct vectors $\vec{x}, \vec{y} \in S$ was determined in \cite{Vinh12} and the case of $\beta (\vec{x}, \vec{y}) = l$ for all $l \in \mathbb{F}_q$ was treated in \cite{Omran2016}.

Let $R$ be a finite local ring of odd characteristic with identity and a non-degenerate symmetric bilinear form $\beta$ over $R^2$. In this paper, we consider a subset $S$ of unimodular vectors in $R^2$ and determine the largest possible size of $S$ so that for any two distinct vectors $\vec{x}, \vec{y} \in S$, $\beta(\vec{x}, \vec{y}) = 0$. This is a generalization of the problem over $\mathbb{F}_q^2$.

Throughout the paper, we assume that all rings have the identity. In Section \ref{Bil}, we review some backgrounds about bilinear forms over $R^n$. We then show the main result when $n = 2$ in Section \ref{Main}. Finally, we conclude the paper with some comments in Section \ref{Con}. 

\section{Bilinear form over finite local rings}\label{Bil}
Let $R$ be a commutative ring and $n$ a positive integer.
A \textit{bilinear form $\beta$ on $R^n$} is a map $\beta:R^n\times R^n\rightarrow R$ such that 
\[
\beta(\vx+\vy,\vz)=\beta(\vx,\vz)+\beta(\vy,\vz)  \text{ and } \beta(r\vx,\vz)=r\beta(\vx,\vz)
\]
and
\[
\beta(\vx,\vz+\vw)=\beta(\vx,\vz)+\beta(\vx,\vw)  \text{ and }
\beta(\vx,s\vz)=s\beta(\vx,\vz).
\]
for all $\vx,\vy, \vz, \vw\in V$ and $r,s\in R$.
Suppose that $\{e_1,e_2,\dots, e_n\}$ is a basis of $R^n$.
For each bilinear form $\beta$ on $R^n$, we have the $n\times n$ associate matrix $B=\big(\beta(e_i,e_j)\big)$.
A bilinear form $\beta$ is said to be \textit{symmetric} if its associate matrix $B$ is  symmetric, and $\beta$ is said to be \textit{non-degenerate} if $B$ is invertible. A \textit{determinant} of a bilinear form $\beta$, denoted by $\det\beta$, is defined to be $\det B$. Two bilinear forms $\beta_1$ and $\beta_2$ over $R^n$ with corresponding matrices $B_1$ and $B_2$ are {\it equivalent} if there exists an invertible matrix $P$ over $R$ such that $B_2 = P^T B_1 P$.

A local ring is a commutative ring with unique maximal ideal.
If $M$ is the unique maximal ideal of a local ring $R$, then the group of units of $R$ is $R\setminus M$. Note that if $u$ is a unit in $R$, then $u + m$ is a unit for all $m \in M$.
In case that $R$ is of odd characteristic, it is shown in \cite{Sri2017} the classifications of non-degenerate symmetric bilinear forms over $R$.

\begin{lem}\label{form}\cite{Sri2017}
Let $R$ be a finite local ring of odd characteristic with unique maximal ideal $M$. Suppose that $\beta$ is a non-degenerate symmetric form on $R^n$, then one of the following holds:

\begin{enumerate}
	\item if $n$ is odd, then $\beta$ is equivalent to
	 \[
	 \beta(\vx,\vy)=x_1y_1-x_2y_2+\dots+x_{n-2}y_{n-2}-x_{n-1}y_{n-1}+u x_ny_n;
	 \]
	 \item if $n$ is even, then $\beta$ is equivalent to
	 \[
	 \beta(\vx,\vy)=x_1y_1-x_2y_2+\dots+x_{n-3}y_{n-3}-x_{n-2}y_{n-2}+x_{n-1}y_{n-1}-ux_ny_n;
	 \]
\end{enumerate}
where $u=1$ or $u$ is a non-square unit in $R$, and $\vx=(x_1,x_2,\dots,x_n), \vy=(y_1,y_2,\dots,y_n)$ are in $R^n$.
\end{lem}

A vector $\vx=(x_1,x_2,\dots,x_n)$ in $R^n$ is said to be \textit{unimodular} if the ideal generated by $x_1,x_2,\dots,x_n$ is equal to $R$.
In particular, if $R$ is a field, then every nonzero vector is unimodular.
For the case $R$ is a finite local ring, we have the following lemma on unimodular vectors.
\begin{lem}\cite{MeePui2013}\label{uni}
	Let $R$ be a finite local ring. Then a vector $\vx=(x_1,x_2,\dots,x_n)\in R^n$ is unimodular if and only if $x_i$ is a unit for some $i\in \{1,2,\dots,n\}$.
\end{lem}

\section{Main result} \label{Main}

For a non-degenerate symmetric bilinear form $\beta$ on $R^n$,
a \textit{unimodular orthogonal set} is a set $S$ of unimodular vectors in $R^n$ in which ${\beta(\vx,\vy)=0}$ for any two distinct vectors $\vx,\vy \in S$.
We denote by $\mathcal{S}(R,n)$ the largest possible cardinality of a unimodular orthogonal set in $R^n$. Here, we only consider $\mathcal{S}(R, 2)$ where $R$ is a finite local ring of odd characteristic.

\begin{thm}\label{main}
	Let $M$ be the maximal ideal of $R$.
	Then
	\begin{align*}
	\mathcal{S}(R,2)=\begin{cases}
	|R|-|M|, &  \det{\beta} \text{ is square;} \\
	2, & \det\beta \text{ is non square}.
	\end{cases}
	\end{align*}
\end{thm}

\begin{proof}
	Assume that $\det\beta$ is a non square unit.
	By Theorem\td\ref{form}, 
	\[
		\beta(\vx,\vy)=x_1y_1-zx_2y_2
	\]
	where $z$ is a fixed non square unit and $\vx=(x_1,x_2), \vy=(y_1,y_2) \in R^2$.
	Let $S$ be a unimodular orthogonal set in $R^2$ and $(a,b)\in S$.
	Since $(a,b)$ is unimodular, $a$ or $b$ is a unit (by Lemma\ref{uni}).
	Suppose that $a$ is a unit.
	Then another vectors in $S$ are of the form $(a^{-1}zby,y)$ where $y$ is a unit in $R$.
	If $\vy_1=(a^{-1}zby_1,y_1),\vy_2=(a^{-1}zby_2,y_2)\in S$, 
	then $\beta(\vy_1,\vy_2)=0$ implies $1=z(a^{-1}b)^2$, so
	$1\in M$ or $z$ is a square unit which is a contradiction.
	We have a similar argument for the case $b$ is a unit.
	Thus, $\mathcal{S}(R,2)\leq 2$.
	Clearly, $\{(1,0),(0,1)\}$ is a unimodular orthogonal set. 
	Therefore, $\mathcal{S}(R,2)= 2$.

	Next, assume that $\det\beta$ is a square unit.
	By Theorem\td\ref{form}, 
	\[
		\beta(\vx,\vy)=x_1y_1-x_2y_2
	\]
	for $\vx=(x_1,x_2), \vy=(y_1,y_2)\in R^2$.
	Clearly, 
	\[
	S=\{(x,x)\mid x\in R\setminus M\}
	\]
	 is a unimodular orthogonal set.
	Then $\mathcal{S}(R,2)\geq |R|-|M|$.
	We show that the converse of the inequality holds.
	Since $\cha(R)$ is odd, $|R|-|M|\geq 2$.
	If $|R|-|M|=2$, then $|R|=3$ and $|M|=1$, i.e., $R=\F_3$.
	By Theorem\td4 of \cite{Omran2016}, $\mathcal{S}(R,2)=2=|R|-|M|$.
	Assume that $|R|-|M|\geq 3$.	
	Let $S$ be a maximal unimodular orthogonal set in $R^2$ and $(a,b)\in S$.
	Since $(a,b)$ is unimodular, $a$ or $b$ is a unit (by Lemma\ref{uni}).
	Suppose that $a$ is a unit.
	Then another vectors in $S$ are of the form $(a^{-1}by,y)$ where $y$ is a unit in $R$, so $|S| \le |R| - |M| + 1$.
	If $\vy_1=(a^{-1}by_1,y_1)$ and $\vy_2=(a^{-1}by_2,y_2)$ are two distinct vectors in $S$. 
	then $\beta(\vy_1,\vy_2)=0$ implies $1=(a^{-1}b)^2$, and so $(a,b)=(a^{-1}b^2,b)$ is also in that form.
	It can be argued similarly for the case that $b$ is a unit.
	Thus, we have $\mathcal{S}(R,2)\leq |R|-|M|$.
	Therefore, the equality holds.	
\end{proof}

\begin{rem}
	From the proof of Theorem\td\ref{main}, we have the following.
	\begin{enumerate}
		\item If $\det\beta$ is square, then all maximal unimodular orthogonal sets of $R^2$ are
		\begin{itemize}
		\item $\{(a,b),(a^{-1}bzy,y)\}$ where $a\in R^\times$, $b\in R$ and $y\in R^\times$, and
		\item  $\{(a,b),(x,{(bz)}^{-1}ax)\}$ where $a\in M$ and $b,x\in R^\times$.
		\end{itemize}
		\item If $\det\beta$ is non-square, then all maximal unimodular orthogonal sets of $R^2$ are
		\[
		\{(ux,x)\mid x\in R^\times \} \text{ or } \{(x,ux)\mid x\in R^\times\}
		\]
		where $u\in R$ with $u^2=1$.
		In particular, if $R=\Z_{p^s}$, then all maximal unimodular orthogonal sets of $R^2$ are
		\[
		\{(x,x)\mid x\in R^\times \}, \{(-x,x)\mid x\in R^\times \} \text{ and } \{(x,-x)\mid x\in R^\times\}.
		\]
	\end{enumerate}
\end{rem}

\section{Concluding remarks}\label{Con}

The problem of finding the largest cardinality of pairwise orthogonal subset in $\mathbb{F}_q^n$ has been solved in \cite{Omran2016,Vinh12}. In Theorem \ref{main}, we solve the similar problem for unimodular vectors in $R^2$ where $R$ is a finite local ring of odd characteristic by using an elementary counting method from the properties of $R$. The problem when $n \ge 3$ could also be considered but it seems to be difficult if we use the method in the proof of Theorem \ref{main} because all equations will be more complicated. Unlike the finite fields case \cite{Omran2016}, the problem when $n$ is odd and $R$ is a general finite local ring is not easy to manage even finding a lower bound for $\mathcal{S}(n, R)$ since $R$ can have zero divisors. We plan to discuss some of these extension works using a new technique in another paper.


\begin{thebibliography}{99}

%
%
%
%
%
\bibitem{Omran2016}
O. Ahmadi, A. Mohammadian, Sets with many pairs of orthogonal vectors over finite fields, {\it Finite Fields Apps.} {\bf 37} (2016), 179--192.

\bibitem{BKT04}
J. Bourgain, N. H. Katz, T. Tao, A sum-product estimate in finite fields, and applications, {\it Geom. Funct. Anal.} {\bf 14} (1) (2004) 27--57.

\bibitem{Erdos1946}
P. Erd{\H o}s, On sets of distances of $n$ point, {\it Am. Math. Mon.} {\bf 53} (1946) 248--250.

\bibitem{Erdos1975}
P. Erd{\H o}s, On some problems of elementary and combinatorial geometry, {\it Ann. Mat. Pura Appl.} {\bf 103} (4) (1975) 99--108.

\bibitem{IR07}
A. Iosevich, M. Rudnev, Erd{\H o}s distance problem in vector spaces over finite fields, {\it Trans. Am. Math. Soc.} {\bf 359} (2007) 6127--6142.

\bibitem{ISX10}
A. Iosevich, I. E. Shparlinski, M. Xiong, Sets with integral distances in finite fields, {\it Trans. Am. Math. Soc.} {\bf 362} (2010) 2189--2204.

\bibitem{MeePui2013} 
Y. Meemark, T. Puirod, Symplectic graphs over finite local rings, {\it European J. Combin.} {\bf 34}(2013), 1114--1124.

\bibitem{Sri2017} 
S. Sriwongsa, Congruence of symmetric inner products over finite commutative rings of odd characteristic, {\it Bull. Aust. Math. Soc.} {\bf 96}(2017), 389--397.

\bibitem{Vinh12}
L. A. Vinh, Maximal sets of pairwise orthogonal vectors in finite fields, {\it Can. Math. Bull.} {\bf 55} (2012) 418--423



\end{thebibliography}
\end{document}